\pdfoutput=1 

\documentclass[12pt]{article} 
\usepackage[utf8]{inputenc} 
\usepackage{authblk}
\usepackage{geometry} 
\usepackage{graphicx} 
\usepackage{booktabs} 
\usepackage{array} 
\usepackage{paralist} 
\usepackage{verbatim} 
\usepackage{subfig} 
\usepackage{multicol}
\usepackage{blindtext}
\usepackage{amssymb}
\usepackage{amsmath}
\usepackage{mathrsfs}
\usepackage{amsthm}
\usepackage{amsfonts}
\usepackage{stmaryrd}
\usepackage{tikz-cd}
\usepackage{cite}

\newcommand{\HdR}{H_{{\rm dR}}}

\newcommand{\Spec}{\text{Spec}}
\newcommand{\del}[1]{\frac{\partial}{\partial#1}}
\newtheorem{theorem}{Theorem}
\newtheorem{definition}{Definition}

\newtheorem{proposition}{Proposition}

\newtheorem{conjecture}{Conjecture}
\newtheorem{remark}{Remark}
\usepackage{hyperref}

\title{Gauss-Manin connection in disguise:\\ Open Gromov-Witten invariants}
\author[1]{Felipe Espreafico}
\affil[1]{{\it Institute of Pure and Applied Mathematics, Rio de Janeiro\/}\\{\tt felipe.espreafico@impa.br}}
\date{} 

\begin{document}

\maketitle
\begin{abstract}
In mirror symmetry, after the work by J. Walcher, the number of holomorphic disks with boundary on the real quintic lagrangian in a general quintic threefold is related to the periods of the mirror quintic family with boundary on two homologous rational curves, known as Deligne conics. Following the ideas of H. Movasati, we construct a quasi-affine space parametrizing such objects enhanced with a frame for the relative de Rham cohomology with boundary at the curves compatible with the mixed Hodge structure. We also compute a modular vector field attached to such a parametrization. 
\end{abstract}

\section{Introduction}

In the 1980s, physicists working on String Theory discovered a phenomenon called Mirror Symmetry: the idea is that the same physical theory can be described by two mathematically different models. Naively, one of these models (the A-model) is mainly related to the symplectic geometry of Calabi-Yau manifolds, while the other one (B-model) is related to the complex algebraic geometry of these spaces. Although many of the tools used by these physicists were not yet rigorous, this led to some impressing purely mathematical results. The most famous example is the quintic threefold: mirror symmetry associates, to a generic quintic threefold in $\mathbb P^4$, a family of manifolds known as mirror quintics. In the paper~\cite{CdlOGP}, the authors used these techniques to make predictions for the number of rational curves on a generic quintic threefold in $\mathbb P^4$. They assumed that the {\it correlation functions\/} from the A-model and the B-model were the same. The first was associated to Gromov-Witten invariants and a prediction to the number of rational curves on a general quintic, and the second was computed explicitly. This correlation function is usually called Yukawa coupling.

In the B-model, the Yukawa coupling is related to a generalization of Kaneko and Zagier's theory of quasi-modular forms developed by Movasati in~\cite{MovArticleMQ}. Movasati's approach relies on the algebraic de Rham cohomology (see~\cite{Grothendieck}) and on the Gauss-Manin connection (see~\cite{katz-oda}). The idea is to consider the moduli space $\sf T_{\rm cl}$ of pairs $(X,[\alpha_1,\dots,\alpha_4])$, where $X$ is a mirror quintic and $[\alpha_1,\dots,\alpha_4]$ is a basis to the third de Rham cohomology of $X$ that, in some sense, respects the Hodge filtration. After computing the Gauss-Manin connection matrix, Movasati was able to define a differential algebra attached to $\sf T_{\rm cl}$ with elements that behave similarly to modular forms: he found relations, between the generators and their derivatives, which generalize the Ramanujan relations we have for the Eisenstein series. The procedure described above is part of a more general program known as Gauss-Manin connection in disguise (GMCD, for short). It is an attempt to construct a general theory of modular forms via functions on a moduli space submitted to a differential equation. Some cases have already been treated, from the elliptic curve case in~\cite{mov-elliptic} to more general cases as in~\cite{AMSY} and the more recent paper~\cite{alimkurylenkovogrin}.

In this text, we will focus solely on the B-side of mirror symmetry and consider another enumerative result: the computation of the numbers of disks with boundary on the real quintic Lagrangian inside the quintic threefold. These numbers were first predicted in~\cite{walcher-opening} and then fully computed in~\cite{pand-wal-sol}. This computation relies on the open Gromov-Witten invariants, which are also computed in these articles. Instead of just considering the mirror family, we need to fix, in each one of the elements of the family, a pair of homologous rational curves $C_{\pm}$ called Deligne conics. We refer the reader to Section~\ref{prelim} for details and to~\cite{mor-wal} for the reasons to consider these two curves. In this context, we have a new natural period to compute, which is the integral of a holomorphic three form over the homology connecting the two curves above. This period satisfies a non homogenous version of the Picard-Fuchs equation, which can be found on Section~\ref{gauss-manin}. 

In order to execute the ideas from the GMCD program, we consider a relative version of the algebraic de Rham cohomology $\HdR^3(X,C_+\cup C_-)$ and a relative version of the Gauss-Manin connection. We define a moduli space $\sf T_{op}$ of triples $(X,C_{\pm}, [\alpha_0,\dots,\alpha_4])$, where $X$ is an element of the mirror family, $C_{\pm}$ is the pair of curves described above on each element of the family and $[\alpha_0,\dots,\alpha_4]$ is a basis of the third relative algebraic de Rham cohomology which respects the {\bf mixed} Hodge structure. For details on mixed Hodge structures, we refer to~\cite{peters-steenbrink} and Section~\ref{MHS and GM} of this text. In the case of the elliptic curve, the Hodge filtration consists of $F^0\supset F^1\supset \dots \supset F^4 =0$ with $\dim F^0 = 5$, $\dim F^1=4$, $\dim F^2 = 2$, $\dim F^3 = 1$ and the weight filtration consists of $0 = W_0\subset W_1\subset \dots W_3$, with $\dim W_0 = \dim W_1 = 0$, $\dim W_2 = 1$ and $\dim W_3 = 5$. These dimensions are computed in Section~\ref{MHS and GM}. The idea of considering mixed Hodge structures and relative cohomology in the GMCD framework was examined in the case of elliptic curves with two fixed points in the paper~\cite{MVC}, in which the authors recovered the theory of Jacobi forms of index zero, and in the paper~\cite{alimkurylenkovogrin}, where the authors considered affine Calabi-Yau varieties.

\begin{definition}\label{relat-enha-mq}
    A relatively enhanced mirror quintic is simply a triple \[(X,C_{\pm},[\alpha_0,\dots,\alpha_4]),\] where $X$ is a mirror quintic, $C_{\pm}$ is the pair of homologous curves cited above and specified in~\eqref{curves} and $[\alpha_0,\dots,\alpha_4]$ is a basis of $H^3_{dR}(X,C_+\cup C_-)$ satisfying the following properties. Let $\delta_0$ be any homology connecting the two curves. Then the properties read:
    \begin{multicols}{2}
\begin{enumerate}[{{i)}}]
\item $\alpha_i\in F^{4-i}\setminus F^{5-i},\quad i >0$;

\item $[\langle\alpha_i,\alpha_j\rangle] =\Phi$;

\item $\alpha_0\in F^1\setminus F^2$;

\item $\alpha_0\in W_2$;

\item $\int_{\delta_0} \alpha_0 = 1$;

\item $\alpha_i\in W_3\setminus W_2,\quad i>0$.
\end{enumerate}
\end{multicols}
where $\Phi=\left[\begin{array}{c
cccc}
0 & 0 & 0 &0 &0\\
0&0 & 0 & 0 & 1 \\
0&0 & 0 & 1 & 0 \\
0&0 & -1 & 0 & 0 \\
0&-1 & 0 & 0 & 0
\end{array}\right]$.
\end{definition}

Although condition $(v)$ above seems to depend on the choice of $\delta_0$, it is actually an algebraic condition which is not influenced by this choice. Indeed, $\alpha_0$ ends up being a class on the image of the map $\HdR^2(C_+\cup C_-)\to \HdR^3(W,C_+\cup C_-)$ (see Section~\ref{MHS and GM}) and, therefore, its integral over any homology class depends only on the boundary of the homology class. 

\begin{theorem}
\label{moduliS} Relatively enhanced mirror quintics can be parametrized by the nine coordinates in affine space given by
\begin{equation}
{\sf T_{op}}:=\Spec\left(\mathbb C\left[s_0,s_1,s_2,s_3,s_4,s_5,s_6,s_7,s_8,\frac1{s_5(s_0^{10}-s_4^{10})s_0s_4}\right]\right).
\end{equation}
\end{theorem}
For an explicit description of the parametrization above, see the proof of Theorem~\ref{moduliS} in Section~\ref{proof1}. With this moduli space in hands, it is possible to compute a differential equation relating these generators $s_i$  with open Gromov-Witten invariants and the virtual count of disks.

\begin{theorem}
\label{Ramanujan}
Consider the space $\sf T_{op}$ defined above. Let $\sf A$ be the Gauss-Manin connection matrix in the basis $\alpha$. There is a unique vector field $\sf R$, for which the connection composed with it is given, in the basis $\alpha$, by
$$
\sf A_{\sf R} = \begin{bmatrix}
0&0&0&0&0\\
0 &0 & 1 & 0 & 0 \\
\sf F&0 & 0 & \sf Y & 0 \\
0&0 & 0 & 0 & -1 \\
0&0 & 0 & 0 & 0
\end{bmatrix},
$$
for regular functions $\sf Y$ and $\sf F$ in $\sf T_{op}$. The expression of $\sf R$, $\sf F$ and $\sf Y$ in the coordinates from Theorem~\ref{moduliS} are
\begin{equation}
\label{Yukawa}
{\sf Y} = \frac{5^{8}\left(s_{4}^{10}-s_{0}^{10}\right)^{2}}{s_{5}^{3}},\quad {\sf F} =- s_{7}{\sf Y},
\end{equation}
\begin{equation}
\label{Ramanuj}
\sf R:\left\{\begin{aligned}
\dot{s}_{0} &=\frac{1}{2s_0s_5}\left(6 \cdot 5^{4} s_{0}^{10}+s_{0}^2 s_3-5^{4} s_4^{10}\right) \\
\dot{s}_{1} &=\frac{1}{s_5}\left(-5^{8} s_{0}^{12}+5^{5} s_{0}^{8} s_1+5^{8} s_{0}^2 s_4^{10}+s_1 s_3\right) \\
\dot{s}_{2} &=\frac{1}{s_5}\left(-3 \cdot 5^{9} s_{0}^{14}-5^{4} s_{0}^{10} s_1+2 \cdot 5^{5} s_{0}^{8} s_2+3 \cdot 5^{9} s_{0}^{4} s_4^{10}+5^{4} s_1 s_4^{10}+2 s_2 s_3\right) \\
\dot{s}_{3} &=\frac{1}{s_5}\left(-5^{10} s_{0}^{16}-5^{4} s_{0}^{10} s_2+3 \cdot 5^{5} s_{0}^{8} s_3+5^{10} s_{0}^{3} s_4^{10}+5^{4} s_2 s_4^{10}+3 s_3^{2}\right) \\
\dot{s}_{4} &=\frac{1}{10s_5}\left(5^{6} s_{0}^{8} s_4+5 s_3 s_4\right) \\
\dot{s}_{5} &=\frac{1}{s_5}\left(-5^{4} s_{0}^{10} s_6+3 \cdot 5^{5} s_{0}^{8} s_5+2 s_3 s_5+5^{4} s_4^{10} s_6\right) \\
\dot{s}_{6} &=\frac{1}{s_5}\left(3 \cdot 5^{5} s_{0}^{8} s_6-5^{5} s_{0}^{6} s_5-2 s_2 s_5+3 s_3 s_6\right)\\
\dot{s}_{7} & = -s_{8}\\
\dot{s}_{8} & = -\frac{5^{12}\left(s_0^{10} - s_4^{10}\right)}{s_5}\cdot \frac{15}8 \left(\frac{s_4}{s_0}\right)^5\frac1{25\sqrt5}\\
\end{aligned}\right. 
\end{equation}
\end{theorem}

The proof of Theorem~\ref{Ramanujan} is computational. The interesting part about it is that, if we consider $s_i$ functions on a variable $q$, take the derivation to be $5q\frac{d}{dq}$, fix initial values $s_{0,0}:=\frac1{\sqrt 5}$, $s_{0,1}:=12\sqrt5$ and $s_{0,4}:=0$ and allow fractional exponents, we get
$$
-5^3{\sf Y}=5 + 2875q + 4876875q^2 + 8564575000q^3 + \dots = \sum_{d=0}^{\infty}n_dd^3\frac{q^d}{1-q^d}
$$
$$
\frac{4}{5^3}{\sf F}(q):=30q^{1/2} + 13800q^{3/2} + 27206280q^{5/2} + 47823842250q^{7/2} + ... = 
$$
$$
 = \sum_{d\, {\rm odd}} n_d^{\rm disk}d^2\frac{q^{d/2}}{1-q^{d}},
$$
	where $n_d$ are the virtual counts of rational curves of degree $d$ on a generic quintic threefold (see~\cite{CdlOGP} and~\cite{MovArticleMQ}) and $n_d^{\rm disk}$ are the virtual counts of disks with boundary on a Lagrangian of a quintic threefold (see~\cite{walcher-opening} and~\cite{pand-wal-sol}). The q-expansions for all functions $s_i$ can be found on the author's webpage\footnote{\url{www.impa.br/~felipe.espreafico/expansionsi}}. Notice that $s_1,s_2,s_3, s_5$ and $s_6$ are the same as the corresponding $t_i$ from~\cite[Theorem 3]{MovArticleMQ} and that $s_0^2 = t_0$ and $s_4^{10} = t_4$.

Our paper is divided in four parts. Section~\ref{prelim} has the aim of presenting basic definitions and properties of the objects and concepts we are using. In Section~\ref{vecfield} we present the proofs of the two main results stated in the introduction. In Section~\ref{periods} we make some computations using the period matrix  and the generalized period domain to explain, for example, why the disk counts are appearing in Theorem~\ref{Ramanujan}. Moreover, these computations help us understand why these functions behave as modular forms.

\section{Preliminaries} \label{prelim}
To start, we present some of the necessary concepts to prove our theorems and fix the notation we are going to use in the rest of the paper.

\subsection{The case of Elliptic Curves}

Before start elaborating on the main topic of the paper, which is to construct a modular framework for open string invariants for the quintic, we recall how a geometric framework is construct for the classical quasi-modular forms. This is an attempt to motivate the constructions in the present text. We refer the interested reader to~\cite{mov-elliptic}, where this framework was first developed, for more details and proofs.

Recall that the algebra of of quasi-modular forms is $\mathbb C[E_2,E_4,E_6]$, where $E_k$ are the Eisenstein series. They satisfy a system of differential equations known as Ramanujan equations. Our idea is to interpret these equations as a vector field on a suitable moduli space of elliptic curves. To do this, we have to deal with elliptic curves and elliptic integrals. In order to deal with both at the same time we will consider enhanced elliptic curves. 

 \begin{definition} A triple $(E, \alpha, \omega)$, where $\alpha$ is a holomorphic 1-form (first piece of the Hodge filtration) and $\omega$ is not holomorphic such that $\langle\alpha,\omega\rangle=1$ is called enhanced elliptic curve.
\end{definition}
 Here, $\langle,\rangle$ is the usual intersection product on the algebraic de Rham cohomology. The definition above allow us to consider, the integrals of $\alpha$ and $\omega$ over paths in $E$, that is, to study elliptic integrals.

\begin{proposition}[\cite{mov-elliptic}, Prop 5.4] The moduli space of enhanced elliptic curves is given by 
\begin{equation}
T= \{(t_1,t_2,t_3)\in\mathbb C^{3}\ |\ 27t_3^2 - t_2^3 \ne0 \},
\end{equation}
where the $(t_1,t_2,t_3)$ corresponds to the triple
$$
E: y^2 = 4(x-t_1)^3 + t_2(x-t_1) +t_3\qquad \alpha = \frac{dx}y\qquad \omega = x\frac{dx}y.
$$
\end{proposition}

We now have a universal family $X\to T$ of enhanced elliptic curves and a basis of sections of the de Rham cohomology bundle. If we compute the Gauss-Manin connection in this basis $(\alpha,\omega)$, we get an explicit matrix in terms of the differentials $dt_i$. 
\begin{proposition}[\cite{mov-elliptic}, Prop. 4.1]
Let $R$ be a vector field in $T$ such that $\nabla_R \alpha = -\omega$ and $\nabla_R\omega = 0$. Then $R$ is unique and it is given by
\begin{equation}
\mathrm{R}=\left(t_{1}^{2}-\frac{1}{12} t_{2}\right) \frac{\partial}{\partial t_{1}}+\left(4 t_{1} t_{2}-6 t_{3}\right) \frac{\partial}{\partial t_{2}}+\left(6 t_{1} t_{3}-\frac{1}{3} t_{2}^{2}\right) \frac{\partial}{\partial t_{3}}
\end{equation}

\end{proposition}

If $R$ is written as a system of differential equations, after multiplying $t_i$ by some constants, we get exactly the Ramanujan equations. By looking at the 1-dimensional locus $L$ for which $R$ generates the tangent space, the maps $t_1,t_2,t_3$ restricted to $L$ will be the Eisenstein series after a change of coordinates. This is a sign that we can generalize modular forms by looking at functions on a suitable moduli space for each case.  

\begin{remark} The locus $L$ in the last paragraph has an interpretation based on the periods (integrals) of $\alpha$ and $\omega$ over integral cycles. From this point of view, $L$ is a fundamental domain for the action of an algebraic group and the natural coordinate for $L$ is given by a quotient of two such periods. For more details, see sections 7 and 8 of~\cite{mov-elliptic}. In the section 4 of the present text, we give this interpretation for the case we are considering.
\end{remark}

\subsection{Mirror Quintic, Deligne Conics and Relative Algebraic de Rham Cohomology}

After the motivation, we introduce our main objects, which are the mirror quintic threefold and the Deligne conics.
\label{alg-derham}
\begin{definition} Let $\psi^5\ne 1$ and let $G$ be the group given by
\label{MQdef}
\begin{equation}
\label{groupG}
G=\left.\left\{\left(a_{0}, \ldots, a_{4}\right) \in \mathbb{Z}_{5}^{5}: \sum_{i} a_{i} \equiv 0 \bmod 5\right\} \middle/ \mathbb{Z}_{5}\right.,
\end{equation}
where $\mathbb{Z}_5$ is embedded diagonally. This group acts on $\mathbb{P}^4$ in the natural way:
$$
(a_0,\ldots,a_4)\bullet [x_0,\ldots,x_4]\mapsto [\mu^{a_0}x_0:\ldots\mu^{a_4}x_4],
$$
where $\mu$ is a primitive fifth root of unit. For us, a {\bf mirror quintic} $X_\psi$ is the resolution of singularities of the quotient
\label{mirror-quintic}
\begin{equation}
\left\{\left[x_{0}: x_{1}: x_{2}: x_{3}: x_{4}\right] \in \mathbb{P}^{4} \mid x_{0}^{5}+x_{1}^{5}+x_{2}^{5}+x_{3}^{5}+x_{4}^{5}-5 \psi x_{0} x_{1} x_{2} x_{3} x_{4}=0\right\} / G.
\end{equation}
\end{definition}
After the quotient and the resolution, one can observe that the varieties obtained will have the Calabi-Yau property. For details, see~\cite{greene-plesser} or~\cite[Section 2.2]{katz-cox}. As we have stated in the introduction, our framework will include two rational curves, the Deligne conics $C_{\pm}\subset X_\psi$. Besides being rational, these curves are homologous as cycles in $H_2(X,\mathbb Z)$. They are defined by the equations
\begin{equation}
C_{\pm}=\left\{ x_{0}+x_{1}=0, x_{2}+x_{3}=0, {x_4}^{2} \pm \sqrt{5\psi}x_{1} x_{3}=0\right\}.
\end{equation}

Our goal now is, instead of considering forms on the absolute cohomology, to consider the relative algebraic de Rham cohomology of the pair $(X_\psi, C_+\cup C_-)$. In order to do this, we recall the definition for the relative algebraic de Rham cohomology in general, which is a generalization of the original definition from~\cite{Grothendieck}.

\begin{definition}
\label{relative-derham-coho}
Let $Y\subset X$ be a closed subvariety of $X$. Define the sheaf $\Omega^m_{X,Y}$ as
$$
\Omega^m_{X,Y}(U) = \Omega^m_X(U)\oplus\Omega^{m-1}_Y(U\cap Y), 
$$
and define a differential operator as 
$$
{\rm d}(\omega,\alpha) = ({\rm d}\omega, \omega|_Y - {\rm d}\alpha),
$$
where the d's appearing on the right-hand side are the boundary operators of $X$ and $Y$. We define the the m-th relative de Rham cohomology of X with boundary on Y, denoted by $H_{\rm dR}^m(X,Y)$ as the m-th-hypercohomology of the complex of sheaves
$$
0\to\Omega^0_{X,Y}\mathop{\to}^d \Omega_{X,Y}^1\mathop{\to}^d\dots\mathop{\to}^d \Omega^m_{X,Y}\to\dots
$$
\end{definition}
With this definition, we get a natural long exact sequence of the pair. In our case of $C_+$ and $C_-$, we get $\dim \HdR^3(X_\psi, C_+\cup C_-)=5$ and a surjection $\HdR^3(X_\psi, C_+\cup C_-)\to\HdR^3(X_\psi)$, which means that we can come up with a basis for $H^3_{dR}(X_\psi,C_+\cup C_-)$ by choosing elements in the pre image of any basis for $\HdR^3
(X_\psi)$ and adding the image of a non zero element from $\HdR^2(C_+\cup C_-)$. Recall that $\dim \HdR^3(X_\psi) =4$.

\subsection{Moduli Space}

\label{moduli1}
After the motivation coming from the Elliptic curve, our goal is to construct an analogous moduli space, we are going to assign coordinates for the space of triples $(X_{\psi},C_{\pm},\omega)$, where $\omega$ is a holomorphic differential 3-form on $X_\psi$. First, we recall that, from~\cite[Section 2.1]{MovArticleMQ}, there are affine coordinates for the pairs $(X, \omega)$: one can associate coordinates $(t_0, t_4)$ to a mirror quintic $X_\psi$, with $\psi^{-5} = \frac{t_4}{t_0^5}$. If we recall $G$ from Definition~\ref{MQdef}, we have:
\begin{align}
\label{eqt_0t_4}
X_{t_{0}, t_{4}} &:=\{f(x)=0\} / G \\
f(x) &:=-t_{4}x_0^5-x_{1}^{5}-x_{2}^{5}-x_{3}^{5}-x_{4}^{5}+5 t_{0}x_0 x_{1} x_{2} x_{3} x_{4}.
\end{align} The 3-form dependent on $(t_0, t_4)$ is the form induced on the resolution of the quotient via the residue form (written in the affine coordinates $x_0=1$, as in~\cite{MovArticleMQ}, Section 2.1)
$$
\omega_{1}:=\frac{d x_{1} \wedge d x_{2} \wedge d x_{3} \wedge d x_{4}}{d f}
$$
which is clearly invariant after the action of the group $G$. Note that these coordinates are only defined $t_0^5\ne t_4$ and $t_4\ne0$. The curves $C_{\pm}$ also depend on these coordinates, but, in order to avoid taking tenth and square roots, we introduce new coordinates $s_0$ and $s_4$ which satisfy $s_0^2 = t_0$ and $s_4^{10} = t_4$. In these coordinates the curves are the resolution of singularities of the quotient of
\begin{equation}
\label{curves}
C_{\pm}=\left\{ s_4^2x_{0}+x_{1}=0, x_{2}+x_{3}=0, s_4{x_4}^{2} \pm \sqrt{5}s_0 x_{1} x_{3}=0\right\} \subset X_{t_0,t_4},
\end{equation}
by the group $G$. In the appendix of~\cite{mor-wal}, they give an explicit way to solve this singularities. We can, therefore, associate a pair $(s_0,s_4)$ to a triple $(X_{s_0,s_4}, C_{\pm}, \omega)$. Of course, this association is only defined when $s_4^{10} \ne s_0^{10}$ and $s_4s_0\ne 0$, since for $s_0=0$ both curves $C_+$ and $C_-$ are equal. We end up with a quasi affine space ${\sf S_{op}}:= \mathbb C^2\setminus\{s_0s_4(s_0^{10} - s_4^{10}) = 0\}$ parametrizing the triples $(X_{s_0,s_4},C_{\pm},\omega)$. This parametrization has an important property which we state below:

\begin{proposition} Let $r\in \mathbb C^{*}$. If $(s_0,s_4)$ is the point corresponding to $(X,\omega, C)\in M$, then $(rs_0,rs_4)$ is the point corresponding to $(X,r^{-2}\omega, C)$.
\label{property}
\end{proposition}
\begin{proof}We know that the isomorphism $(x_0,x_1,x_2,x_3,x_4)\mapsto (x_0,rx_1,rx_2,rx_3,rx_4)$ between $X_{s_0,s_4} = X_{t_0,t_4}$ and $X_{rs_0,rs_4} = X_{r^2t_0,r^{10}t_4}
$ takes $\omega(t_0,t_4)$ to $r^{-2}\omega(r^2t_0,r^{10}t_4)$ (see~\cite{MovArticleMQ}, sec. 2.1). Using that $t_0 = s_0^2$ and $t_4 = s_4^{10}$, we have our result. We just need to check that it maps the curve $C$ to its correspondent. Indeed, we have
$$
C_{rs_0,rs_4} = \left\{r^2s_4^2x_{0}+x_{1}=0, x_{2}+x_{3}=0, rs_4x_{4}^{2} + rs_0\sqrt{5} x_{1} x_{3}=0\right\} \subset X_{rs_0,rs_4}
$$ 
and
$$
C_{s_0,s_4} = \left\{s_4^2x_{0}+x_{1}=0, x_{2}+x_{3}=0, s_4x_{4}^{2}+s_0\sqrt{5} x_{1} x_{3}=0\right\} \subset X_{s_0,s_4}.
$$
Taking a point of the second curve and applying the isomorphism, we have that the equations of the first one are satisfied. 
\begin{align*}
&r^2s_4^2x_{0}+(r^2x_{1})= r^2(s_4x_0 + x_1) = 0,\\
&(r^2x_2) + (r^2x_3) = r^2(x_2 + x_3) = 0,\\
 &rs_4(r^2x_{4})^{2} + rs_0\sqrt{5} (r^2x_{1})(r^2 x_{3})= r^5(s_4x_{4}^{2}+s_0\sqrt{5} x_{1} x_{3})=0.
\end{align*}
This finishes the proof.
\end{proof}

Considering the zero locus of the equation $f$ from~\eqref{eqt_0t_4} in the product $\mathbb P^4\times \sf S_{op}$, we get a family $\mathcal X\to \sf S_{op}$. Now, $C_{\pm}$ induce a subfamily $\mathcal Y\subset \mathcal X$ corresponding to the curves. The next step is to define a structure on this triple which encodes properties of the relative algebraic de Rham cohomology of each element on the family. 

\subsection{Relative Gauss-Manin Connection}
\label{gauss-manin}
One of the most important tools in our paper is the algebraic Gauss-Manin connection. It was first introduced in the paper~\cite{katz-oda}, by Katz and Oda. They defined the Gauss-Manin connection for any smooth morphism $\pi: X\to S$ of smooth schemes (in other words, a family over $S$). Roughly, it formalizes the notion of differentiating with respect to the parameters of a family. They consider the algebraic de Rham cohomology of the family, which is the hyperderived sheaf 
\begin{equation}
\label{adrc-sheaf}
\mathcal{H}_{dR}^m(X/S):=R^m\pi_*(\Omega^{\bullet}_{X/S})
\end{equation}
defined over $S$. This is the sheaf associated to the presheaf $U\mapsto \mathbb H^m\left(\pi^{-1}(U), \Omega^{\bullet}_{X/S}\right)$, where $\mathbb H$ denotes the hypercohomology. They define a connection
\begin{equation}
\nabla:\mathcal{H}_{dR}^m(X/S)\to \mathcal{H}_{dR}^m(X/S)\otimes \Omega^1_S
\end{equation}
by considering a filtration on $\Omega^{\bullet}_{X}$ and taking the spectral sequence associated to this filtration. Here, we need a relative version of this construction. If we have a smooth closed subvariety $Y\subset X$ for which the restriction of $\pi$ is smooth, it is not difficult to make the same construction and get a connection
\begin{equation}
\label{gmrel}
\nabla:\mathcal{H}_{dR}^m(X, Y/S)\to \mathcal{H}_{dR}^m(X,Y/S)\otimes \Omega^1_S,
\end{equation}
where $\mathcal{H}_{dR}^m(X, Y/S)$ is defined as above, by considering $\Omega^{\bullet}_{X,Y/S}$ and not $\Omega^{\bullet}_{X/S}$ (see Definition~\ref{relative-derham-coho}, but considering $\Omega^{\bullet}_{X/S}$ and $\Omega^{\bullet}_{Y/S}$ instead of $\Omega^{\bullet}_{X}$ and $\Omega^{\bullet}_{Y}$). This definition of the relative version of the Gauss-Manin connection is compatible with the long exact sequence of the pair and thus many computations from the absolute case can be transported to the relative case. The most important example would be the equality
\begin{equation}
\label{commute}
d\left(\int_{\delta}\omega\right) = \int_{\delta}\nabla\omega,
\end{equation}
where $\delta$ is any cycle (absolute or relative). In order to compute the connection for the family $\sf S_{op}$, we will need differential relations among the periods and use~\eqref{commute}. We consider, therefore, the non-homogenous version of the Picard-Fuchs equation 
\begin{equation}
\label{PFIH}
\theta^{4}-z\left(\theta+\frac{1}{5}\right)\left(\theta+\frac{2}{5}\right)\left(\theta+\frac{3}{5}\right)\left(\theta+\frac{4}{5}\right)=15 \frac{\sqrt{5^{-5} z}}{8}
\end{equation}
satisfied by the integral of a holomorphic three-form on $X_\psi$ over the homology connecting the two curves $C_{\pm}$. Above, $z = \psi^{-5}$  and $\theta = z\del{z}$. The equation in the form above is in~\cite[page 1170]{pand-wal-sol} using coordinates $t$ with $z = 5^{5}e^t$. Observe that if we consider the right-hand side of~\eqref{PFIH} to be zero, we get the classical equation appearing on~\cite{CdlOGP} for which the periods given by integrals of the holomorphic three form over absolute homology classes are solutions.

\subsection{Mixed Hodge structure on the relative algebraic de Rham cohomology}
\label{MHS and GM}
Consider the family $\mathcal X\to \sf S_{op}$ defined in Section~\ref{moduli1}. We want to study the structure we have on the relative algebraic de Rham cohomology. Recall from algebraic topology that we can define the cup product in the relative cohomology by restricting the cup product from absolute cycles to relative ones. In the case we are considering, i.e., the cohomology $\HdR^3(X,Y)$, where $X$ is a mirror quintic and $Y$ is the union of the two curves $C_{\pm}$, we have that the element coming from $\HdR^2(Y)$ is degenerate for the product. Therefore, we can use the computation from~\cite[Section 2.3]{MovArticleMQ}, since the intersection matrix will be the same after adding zeros to the first column and row.

The Mixed Hodge structure on the relative cohomology is defined by means of the {\it mixed cone} (see~\cite[Theorem 3.22 and Example 3.24]{peters-steenbrink}). The important part is that the long exact sequence of the pair ends up being a long exact sequence of mixed Hodge structures and so the Hodge and the weight filtration may be defined via this sequence. Also, the Gauss-Manin connection satisfies the Griffths transversality property for the Hodge filtration, that is, $\nabla(F^p) \subset F^{p-1}\otimes\Omega^1_S$. This is easily seen using that fact that the Gauss-Manin connection commutes with the long exact sequence and that the long exact sequence is a sequence of mixed Hodge structure. We have a nice description of the mixed Hodge structure for our case, given by the proposition below.

\begin{proposition}
\label{str-MHS}
We have that the image of the map $\alpha: H_{\rm dR}^2(Y) \to H_{\rm dR}^3(X,Y)$ defined via the long exact sequence is one-dimensional and it is contained in $F^1H^3_{\rm dR}(X,Y)\setminus F^2H^3_{\rm dR}(X,Y)$, where $F$ represents the Hodge filtration. Also this image is exactly the second piece of the weight filtration $W_2(H^3_{\rm dR}(X,Y))$. In particular, any generator of this image satisfy properties (iii) and (iv) of Definition~\ref{relat-enha-mq}. 
\end{proposition}
\begin{proof} 
Recall that, as $Y$ is the union of two $\mathbb P^1$, it has $F^2H^2_{\rm dR}(Y) = 0$ and $F^1H^2_{\rm dR}(Y) = H^2_{\rm dR}(Y)$. Therefore, as the long exact sequence of the pair is a sequence of mixed Hodge structures, we conclude that the image is contained in $F^1$ and not on $F^2$ (since we have $\alpha(F^{p}(H^2_{\rm dR}(Y))) = {\rm Im}(\alpha)\cap F^pH^3_{\rm dR}(X,Y)$). For the part about the weight filtration, we just need to use that the weight filtration for $H^2_{\rm dR}(Y)$ is given by $W_0 = 0$, $W_1 = 0$ and $W_k = H^2_{\rm dR}(Y)$ for $k\geq 2$. Therefore, by again using that the exact sequence is a sequence of mixed Hodge structures, we conclude that the image of $\alpha$ is $W_2 H^3_{\rm dR}(X,Y)$. 
\end{proof}

\section{Proofs of the main theorems}
\label{vecfield}

In this section, our goal is to prove the two main theorems stated in the introduction. For simplicity, throughout this section, we denote $Y=C_+\cup C_-$.

\subsection{Proof of Theorem~\ref{moduliS}}
\label{proof1}

\begin{proof}[Proof of Theorem~\ref{moduliS}]

Consider the basis $W=\{\omega_1,\dots,\omega_4\}$ for $H^3_{\rm dR}(X)$, with $\omega_1$ a holomorphic 3-form and $\omega_i := \nabla_{\frac{\partial}{\partial t_0}}(\omega_{i-1})$, where $t_0:=s_0^2$ and $\nabla$ is the Gauss-Manin connection on the absolute cohomology (see Section~\ref{moduli1}). Notice that we use derivatives with respect to $t_0$ instead of $s_0$, since it makes it easier to compare with the absolute case. Of course, we can go from $t_0$ to $s_0$ via the relation
$$
\frac{\partial}{\partial t_0 } = \frac1{2s_0}\frac{\partial}{\partial s_0}.
$$

Using that the map $H_{dR}^3(X,Y)\to H_{dR}^3(X)$ is surjective, we can take elements $\omega_1,\dots,\omega_4\in H_{dR}^3(X,Y)$ corresponding to the basis $W$. Then, we choose a generator of ${\rm Im}(\HdR^2(Y)\to \HdR^3(X,Y))$ and call it $\omega_0$. This generator is chosen to be the image of the of the Poincaré dual of the difference of the homology classes $[C_+]$ and $[C_-]$. In this way, we get that the integral of $\omega_0$ over the homology connecting the two curves is 1. Now, consider the matrix:
\begin{equation}
\label{S}
S=\begin{bmatrix}
1 & 0 &0 &0 & 0\\
0& 1 & 0 & 0 & 0 \\
0&a & b & 0 & 0 \\
s_{7} &c & s_{6} & s_{5} & 0 \\
s_{8} &s_{1} & s_{2} & s_{3} & d
\end{bmatrix}
\end{equation}
and assume it is invertible, which implies $s_5\ne0$.
The basis $\alpha = S\omega$ satisfy all properties on Definition~\ref{relat-enha-mq} except for (ii). Indeed, by Proposition~\ref{str-MHS} above, (iii) and (iv) are satisfied and, as the map $H^3(X,Y)\to H^3(X)$ preserves filtrations and the Gauss-Manin connection sends $F^i$ to $F^{i-1}$, we have condition (i). Condition (v) is satisfied by construction. Demanding $S[\langle \omega_i,\omega_j\rangle]S^{\rm tr} = \Phi$, that is, condition (ii), we get equations relating $a,b,c,d$ and the other variables:
\begin{align}
cb-s_{6}a &=3125 s_{0}^{6}+s_{2}, \\
d &=-bs_{5}, \\
s_{5} a &=-3125 s_{0}^{8}-s_{3}, \\
d &=625\left(s_{4}^{10}-s_{0}^{10}\right).
\end{align}
To perform this computation, we make use of the intersection product computed in~\cite[Proposition 3]{MovArticleMQ} and the fact that $\alpha_0$ is degenerate for the intersection product (see Section~\ref{MHS and GM}). This relations imply that we can drop the variables $a$, $b$, $c$ and $d$ and only consider, besides $s_0$ and $s_4$, five coordinates $s_1,s_2,s_3,s_5,s_6$, which are the same as the corresponding $t_i$ in~\cite{MovArticleMQ}, and the extra two coordinates $s_7$ and $s_8$ which only appear in the relative case. Notice that, for each matrix $S$, we obtain a different basis $\alpha$ and for each basis $\alpha$, we obtain a matrix by inverting $S$ and solving $\omega = S^{-1}\alpha$.
\end{proof}

\subsection{Proof of Theorem~\ref{Ramanujan}}
\label{proof2}
To prove Theorem~\ref{Ramanujan}, we need first to compute the Gauss-Manin connection in the basis $\alpha$. Fix $z=\frac{t_4}{t_0^5} = \frac{s_4^{10}}{s_0^{10}}$ and consider the non-homogenous Picard-Fuchs differential equation~\ref{PFIH}. Let $\eta_1 = t_0\omega_1$ and $\eta_0 = \omega_0$. Those forms are the ones we get by pulling back $\omega_0$ and $\omega_1$ via the isomorphism $X_{1,\frac{s_4}{s_0}}\cong X_{s_0,s_4}$ (see Proposition~\ref{property}). Define $\eta_i = \nabla_{\frac{\partial}{\partial z}}(\eta_{i-1})$. By the definition of $z$, we get a relation between $\frac{\partial}{\partial z}$ and $\frac{\partial}{\partial t_0}$ given by
\begin{equation}
\frac{\partial}{\partial z}=\frac{-1}{5} \frac{t_{0}^{6}}{t_{4}} \frac{\partial}{\partial t_{0}}.
\end{equation}
It is easy therefore to get a relationship between the basis $\eta$ and $\omega$. We call this matrix $C$.
As we observed, the Picard-Fuchs equation~\eqref{PFIH} is satisfied by the integral of $\eta_1$ over the homology connecting the curves $C_{+}$ and $C_{-}$. Using that $\int \eta_0 = \int \omega_0 = 1$, the fact that $\eta_1$ satisfy~\eqref{PFIH} implies the following equality:
\begin{multline}
\int_{\delta_0}\nabla_{\frac{\partial}{\partial z}}
\eta_4=\frac{-p}{z^4(z-1)}\int_{\delta_0}\eta_0 + \frac{-24}{625 z^{3}(z-1)} \int_{\delta_0}\eta_1 +\\
+\frac{-24 z+5}{5 z^{3}(z-1)} \int_{\delta_0}\eta_2+\frac{-72 z+35}{5 z^{2}(z-1)} \int_{\delta_0}\eta_3+\frac{-8 z+6}{z(z-1)}\int_{\delta_0}\eta_4.
\end{multline}
By comparing the integrands, we can see that the Gauss-Manin matrix in the basis $\eta$ is given by:
\begin{equation}
{\sf B_1} = \begin{bmatrix}
0&0&0&0&0\\
0&0&1&0&0\\
0&0&0&1&0\\
0&0&0&0&1\\
a_0&a_1&a_2&a_3&a_4\\
\end{bmatrix}dz,
\end{equation}
where $a_i$ are the coefficients of $\frac{\partial^i}{\partial z^i}$ in~\ref{PFIH}.
To end, we compute
\begin{equation}
{\sf B_2}:= (dC + C\cdot {\sf B_1})C^{-1},
\end{equation}
which is the Gauss-Manin connection written in basis $\omega$. Observe that the submatrix formed by rows and columns from 2 to 5 is the Gauss-Manin connection in $H_{dR}^3(X)$ as computed in~\cite{MovArticleMQ} Section 2.6.  To compute the matrix in the basis $\alpha$ from Theorem~\ref{moduliS}, we compute:
\begin{equation}
\label{GM}
{\sf A} = (dS + S\cdot {\sf B_2})S^{-1},
\end{equation}
where $S$ is given in~\eqref{S}.

\begin{proof}[Proof of Theorem~\ref{Ramanujan}] We take an unknown vector field $\sf R$ and let its first six coordinates be equal to the ones in~\cite[Theorem 3]{MovArticleMQ}. As the 4x4 submatrix of the Gauss-Manin Connection is the same as in the absolute case, by direct computation, we end up with
\begin{equation}
\left[\begin{smallmatrix}
0 & 0 & 0 & 0 & 0\\
0 & 0 & 1 & 0 & 0\\
\frac{5^8\left(s_0^{10} - s_4^{10}\right)^2}{s_5^3}s
_7 & 0 & 0 & \frac{5^{8}\left(s_{4}^{10}-t_{0}^{10}\right)^{2}}{s_{5}^{3}} & 0 \\
ds_{7}({\sf R}) +s_8 &0 &0&0&-1\\
ds_{8}({\sf R}) + \frac{5^{12}\left(s_0^{10} - s_4^{10}\right)}{s_5}p & 0 & 0& 0& 0\\
\end{smallmatrix}\right]
\end{equation}
after plugging $\sf R$ in the matrix $\sf A$ from~\eqref{GM}. Now, recalling that ${\rm d}s_i(\sf R)$ is the i-th coordinate of the vector field $\sf R$ and that the first column has to have only zeros except for the third line, we can determine the other coordinates of $\sf R$ uniquely. This gives us the desired vector field and ends the theorem.
\end{proof}
\section{Relationship with Periods}
\label{periods}

In this section, our goal is to explain why are the functions $\sf Y$ and $\sf F$ appearing. For this, we need to look at the period domain of the space $\sf T_{op}$ from Theorem~\ref{moduliS}. For us, a period is simply a number obtained by integration of differential forms over cycles in homology. Here, we are specially interested in the integration of 3-forms over 3 dimensional cycles. Consider a symplectic basis of the homology group $H_3(X)$ given by $\{\delta_1,\delta_2,\delta_3,\delta_4\}$. Also, let $\delta_0$ be the homology connecting the two rational curves $C_+$ and $C_-$. Of course, those five homology classes form a basis for $H_3(X,Y)$.
\begin{definition}
The period matrix is defined as
\begin{equation}
\label{periodmatrix}
{\sf P} = [p_{ij}] = \left[\int_{\delta_i}\alpha_j\right]_{ij},
\end{equation}
where the $\alpha_j$ form a basis satisfying the conditions from~\ref{relat-enha-mq}.
\end{definition}
Using Poincaré duality, one can easily see that this matrix is related to the intersection matrix of the $\alpha_i$'s by the formula
\begin{equation}
\label{relationsperiods}
\left[\left\langle\alpha_{i}, \alpha_{j}\right\rangle\right]=\left[\int_{\delta_{i}} \alpha_j\right]^{\mathrm{T}} \Psi^{-\mathrm{T}}\left[\int_{\delta_{i}} \alpha_{j}\right],
\end{equation}
where $\Psi$ is the intersection matrix of the basis $\delta$, which is given by
$$
\Psi:=\left[\begin{array}{ccccc}
0 & 0 & 0 & 0& 0\\
0&0 & 0 & 1 & 0 \\
0&0 & 0 & 0 & 1 \\
0&-1 & 0 & 0 & 0 \\
0&0 & -1 & 0 & 0
\end{array}\right].
$$
\begin{definition}
Let $\sf G$ be the group given by:
\begin{equation}
{\sf G}:=
\left\{g=\left(\begin{array}{ccccc}
1 & 0 &0 &h_1 &h_2\\
0 &g_{11} & g_{12} & g_{13} & g_{14} \\
0 & 0 & g_{22} & g_{23} & g_{24} \\
0 & 0 & 0 & g_{33}& g_{34} \\
0 &0 & 0 & 0 & g_{44}\\
\end{array}\right), \quad h_k,g_{ij}\in\mathbb C\quad \begin{array}{c}
g_{11}g_{44}=1,\\
g_{22}g_{33} = 1,\\
g_{12}g_{44} + g_{22}g_{34} = 0,\\
g_{13}g_{44}+g_{23}g_{34}-g_{24}g_{33} = 0.
\end{array}\right\}
\end{equation}
\end{definition}
The group ${\sf G}$ acts in an element $(X,\alpha)$ in the moduli space by the right as $(X,\alpha)\bullet g = (X,\alpha g)$ where $\alpha$ is seen as a row vector.  Considering the relations, we can write this group in terms of six "g"-coordinates and two "h"-coordinates, as below:
\begin{equation}
\label{coordinatesG}
(g_1,g_2,g_3,g_4,g_5,g_6,h_1,h_2) = \left(\begin{array}{ccccc}
1 & 0 & 0 & h_1 & h_2 \\
0 & g_{1}^{-2} & -g_{3} g_{1}^{-1} & \left(-g_{3} g_{6}+g_{4}\right) g_{1}^{-2} & \left(-g_{3} g_{4}+g_{5}\right) g_{1}^{-2} \\
0 & 0 & g_{2}^{-1} & g_{6} g_{2}^{-1} & g_{4} g_{2}^{-1} \\
0 & 0 & 0 & g_{2} & g_{2} g_{3} \\
0 & 0 & 0 & 0 & g_{1}^2
\end{array}\right).
\end{equation}
Notice that our coordinate $g_1$ is different from~\cite{MovArticleMQ}: ours is the square root of the one in that article.

\begin{proposition} The action of ${\sf G}$ written on the coordinates $s_i$ of $\sf T_{op}$ is
\label{actiong}
$$
\begin{aligned}
&g \bullet s_{0}=s_{0} g_{1} \\
&g \bullet s_{1}=s_{1} g_{1}^{2}+c g_{1} g_{2} g_{3}+a g_{1} g_{2}^{-1} g_{4}-g_{3} g_{4}+g_{5} \\
&g \bullet s_{2}=s_{2} g_{1}^{3}+s_{6} g_{1}^{2} g_{2} g_{3}+b g_{1}^{2} g_{2}^{-1} g_{4} \\
&g \bullet s_{3}=s_{3} g_{1}^{4}+s_{5} g_{1}^{3} g_{2} g_{3} \\
&g \bullet s_{4}=s_{4} g_1 \\
&g \bullet s_{5}=s_{5} g_{1}^{3} g_{2} \\
&g \bullet s_{6}=s_{6} g_{1}^{2} g_{2}+b g_{1}^{2} g_{2}^{-1} g_{6}\\
&g\bullet s_{7} = s_{7}g_2 + h_1\\
&g\bullet s_{8} = s_{7}g_2g_3 + s_{8}g_1 +h_2\\
\end{aligned}
$$
where $a,b,c$ are the expressions in given in terms of the coordinates $s_i$ from~\eqref{S}.
\end{proposition}
\begin{proof}
We start with a pair $(X_{s_0,s_4},\omega_1)$. This form $\omega_1$, together with its derivatives and a form $\omega_0$ yields a basis of $H^3(X,Y)$. Multiplying by the matrix $S$ from equation~\eqref{S} we get a basis satisfying the conditions in Definition~\ref{relat-enha-mq} depending on the coordinates $s_i$. Now, let $g\in G$ act. By definition,  $\alpha_1 = \omega_1$ would be multiplied by $g_1^{-2}$. In order to write the new element of the moduli space in coordinates, we need to normalize $\omega_1$. Consider the form $g_1^2\omega_1$ in the beginning. After this change, we need to multiply the basis $\omega$ by the matrix 
$$
K = \begin{pmatrix}
1 & 0 & 0 & 0 &0\\
0 & k & 0 & 0 &0\\
0 & 0 & k^2 & 0 & 0\\
0 & 0 & 0 & k^3 & 0\\
0 & 0 & 0 & 0 & k^4\\
\end{pmatrix}
$$
where $k = g_1^2$. This is because of the other forms of the basis $\omega$ are derivatives of $\omega_1$.

 Notice that, by doing this, we are considering the point $(ks_1,ks_4)$ of the moduli space of mirror quintics enhanced with a holomorphic 3-form and two ration curves. Now, the matrix $g^TSK$ takes the basis $\omega_0,\omega_1,\dots,\omega_4$ to its image by the action of $g$. The entries of this matrix are the coordinates of the image. For example, the entry $(5,2)$ should be the coordinate $t_1$, etc. After completing this computation, we get the result.
\end{proof}

\subsection{The $\tau$-matrix}

We want to consider the orbits of the action of ${\sf G}$ on $\sf T_{op}$ and their images by the period map. For this, we notice that ${\sf G}$ acts on the space of matrices by right-multiplication. This action clearly preserves the relations~\eqref{relationsperiods} and is compatible with the action on ${\sf T_{op}}$, in the sense that the period matrix relative to a basis $\alpha\bullet g$ is simply $Ag$, where $A$ is the matrix with respect to $\alpha$.

\begin{proposition}
\label{tau}
For any period matrix $\sf P$ satisfying the relations~\eqref{relationsperiods}, there exists a unique $g\in G$ such that ${\sf P}g$ can be written on the form 
\begin{equation}
\label{taumatrix}
\tau = \begin{pmatrix}
1 & \tau_4 & \tau_5 & 0 & 0 \\
0 & \tau_0 & 1 & 0 & 0 \\
0 & 1 & 0 & 0 & 0\\
0& \tau_1 & \tau_3 & 1 & 0\\
0& \tau_2 & -\tau_0\tau_3 +\tau_1& -\tau_0 & 1
\end{pmatrix}.
\end{equation}
for some $\tau_i$.
\end{proposition}

\begin{proof}Write $g$ in the form~\eqref{coordinatesG}. Multiplying $g$ by a general matrix $\sf P$ and using the relations~\eqref{relationsperiods}, we get a matrix of the form
$$
{\sf P}g =\left(\begin{array}{ccccc}
1&*&*&*&*\\
0&* & 1 & 0 & * \\
0&1 & 0 & 0 & 0 \\
0&* & * & * & * \\
0&* & * & * & *
\end{array}\right)
$$
Using the computations done in~\cite[Section 3.3]{MovArticleMQ}, writing down the equations for ${\sf P}g=\tau$ when the entries in $\tau$ are independent of $\tau_i$, we get that the first coordinates of $g':=g^{-1}$ are necessarily given by
$$
\begin{aligned}
&(g'_{1})^2={\sf P}_{21}^{-1}, \\
&g'_{2}=\frac{-{\sf P}_{21}}{{\sf P}_{11} {\sf P}_{22}-{\sf P}_{12} {\sf P}_{21}}, \\
&g'_{3}=\frac{-{\sf P}_{22}}{{\sf P}_{21}}, \\
&g'_{4}=\frac{-{\sf P}_{12} {\sf P}_{23}+{\sf P}_{13} {\sf P}_{22}}{{\sf P}_{11} {\sf P}_{22}-{\sf P}_{12} {\sf P}_{21}}, \\
&g'_{5}=\frac{{\sf P}_{11} {\sf P}_{22} {\sf P}_{24}-{\sf P}_{12} {\sf P}_{21} {\sf P}_{24}+{\sf P}_{12} {\sf P}_{22} {\sf P}_{23}-{\sf P}_{13} {\sf P}_{22}^{2}}{{\sf P}_{11} {\sf P}_{21} {\sf P}_{22}-{\sf P}_{12} {\sf P}_{21}^{2}}, \\
&g'_{6}=\frac{{\sf P}_{11} {\sf P}_{23}-{\sf P}_{13} {\sf P}_{21}}{{\sf P}_{11} {\sf P}_{22}-{\sf P}_{12} {\sf P}_{21}} \\
\end{aligned}
$$
It suffices to compute $h'_1$ and $h'_2$. After computing ${\sf P}g$ = ${\sf P}g'^{-1}$, we get:
\begin{equation}
\label{matrixtau}
\tau=\left(\begin{array}{ccccc}
1 & \frac{{\sf P}_{01}}{{\sf P}_{21}} & \frac{{\sf P}_{01}{\sf P}_{22} - {\sf P}_{02}{\sf P}_{21}}{{\sf P}_{11}{\sf P}_{22} - {\sf P}_{12}{\sf P}_{21}} & P & Q \\ 
0 & \frac{{\sf P}_{11}}{{\sf P}_{21}} & 1 & 0 & 0 \\
0 & 1 & 0 & 0 & 0 \\
0 & \frac{{\sf P}_{31}}{{\sf P}_{21}} & \frac{-{\sf P}_{21} {\sf P}_{32}+{\sf P}_{22} {\sf P}_{31}}{{\sf P}_{11} {\sf P}_{22}-{\sf P}_{12} {\sf P}_{21}} & 1 & 0 \\
0 & \frac{{\sf P}_{41}}{{\sf P}_{21}} & \frac{-{\sf P}_{21} {\sf P}_{42}+{\sf P}_{22} {\sf P}_{41}}{{\sf P}_{11} {\sf P}_{22}-{\sf P}_{12} {\sf P}_{21}} & -\frac{{\sf P}_{11}}{{\sf P}_{21}} & 1
\end{array}\right)
\end{equation}
where $P$ and $Q$ depend linearly on $h'_1$, $h'_2$. Making $P=Q=0$, we find expressions for $h_1$ and $h_2$. This shows existence and uniqueness.
\end{proof}

Define ${\sf L_{op}}$ as the locus in the moduli space $\sf T_{op}$ from Theorem~\ref{moduliS} for which the period matrix is of the form~\eqref{taumatrix}. Our goal is to express the functions $s_i$ restricted to this locus in some coordinate. To do this, we first consider the points of $\sf T_{op}$ of the form $(1,0,0,0,y,1,0,0,0)$, where $y^{10} = z$ (the same coordinate used in Section~\ref{proof2}). Then, we compute the period matrix $\sf P$ for these points and find the elements $g\in \sf G$ for which $\sf P$ is of the form~\eqref{taumatrix}. Then, by computing the elements $(1,0,0,0,y,1,0,0,0)\bullet g$, we will get expressions for the coordinates of $\sf L_{op}$. We consider the periods
$$
x_{ij} = \int_{\delta_j}{\tilde{\eta_i}},
$$
where $\tilde{\eta}_0 = \omega_0 = \alpha_0$, $\tilde{\eta_1}$ is the holomorphic three form associated to the point $(1,\frac{s_4}{s_0})$ of the moduli space of triples $(X,\omega,C_{\pm})$ defined in Section~\ref{moduli1} and $\tilde{\eta_i} = \theta(\tilde{\eta_{i-1}})$ (recall $\theta = z\frac{\partial}{\partial z}$). These periods are related to the solutions of the Picard-Fuchs equation via the matrix
\begin{equation}
\left(\begin{array}{l}
x_{01}\\
x_{11} \\
x_{21} \\
x_{31} \\
x_{41}
\end{array}\right)=\left(\begin{array}{ccccc}
\frac1{2\pi^2} & 0 & 0 &\frac{5^4}{2\cdot(2\pi i)^2}&\frac{5^4}{4\cdot(2\pi i)^3}\\
0 & 0 & 0 & 1 & 0 \\
0 & 0 & 0 & 0 & 1 \\
0 & 0 & 5 & \frac{5}{2} & -\frac{25}{12} \\
0 &-5 & 0 & -\frac{25}{12} & 200 \frac{\zeta(3)}{(2 \pi i)^{3}}
\end{array}\right)\left(\begin{array}{c}
\varphi\\
\frac{1}{5^{4}} \psi_{3} \\
\frac{2 \pi i}{5^{4}} \psi_{2} \\
\frac{(2 \pi i)^{2}}{5^{4}} \psi_{1} \\
\frac{(2 \pi i)^{3}}{5^{4}} \psi_{0}
\end{array}\right),
\end{equation}
where $\psi_i$ are solutions for the homogenous equation as it is presented in~\cite[Introduction]{MovArticleMQ} and $\varphi$ is the solution for inhomogeneous equation~\eqref{PFIH} given by the series:
$$
2\sum_{m\,{\rm odd}}^{\infty} \frac{(5m)!!}{(m!!)^5}(5^{-5}z)^{m/2}
$$
where the double exclamation point mean we multiply all the odd numbers less or equal to the number (e.g $7!! = 1\cdot 3\cdot5\cdot 7 = 105$). The expression for $x_{01}$ is taken from~\cite{mor-wal} and~\cite{pand-wal-sol}. Notice that the notation for the series above in~\cite{mor-wal} is different: they take $\tau = \frac{\varphi}{30}$. The expressions for the other periods are taken from~\cite[Introduction]{MovArticleMQ}. To find the period matrix in terms of $x_{ij}$, we need to change from $\tilde{\eta}$ to the basis $\alpha$ from the moduli space.  We consider, therefore, the matrices
$$
S=\left(\begin{array}{ccccc}
1 & 0 & 0 & 0 & 0\\
0 & 1 & 0 & 0 & 0 \\
0 &-5^{5} & -5^{4}(z-1) & 0 & 0 \\
0 & -\frac{5}{z-1} & 0 & 1 & 0 \\
0 & 0 & 0 & 0 & 5^{4}(z-1)
\end{array}\right)
$$
and
$$
 T=\left(\begin{array}{ccccc}
1 &0 & 0&0 &0\\
0&1 & 0 & 0 & 0 \\
0&-1 & -5 & 0 & 0 \\
0&2 & 15 & 25 & 0 \\
0&-6 & -55 & -150 & -125
\end{array}\right),
$$where $T$ takes the basis $\eta$ to the basis $\omega$ and $S$ takes $\omega$ to $\alpha$. Notice that the matrices above have already been used in Section~\ref{vecfield}, but for general $s_i$. Therefore, the period matrix is simply ${\sf P}=[x_{ij}] (ST)^{T}$. Now, using Proposition~\ref{tau} and Proposition~\ref{actiong}, it is easy to find $g\in \sf G$ for which ${\sf P}g$ is of the form~\eqref{taumatrix} and compute the action of $g$ on $(1,0,0,0,y,1,0,0,0)$ to get the $s_i$ restricted to $\sf L_{op}$. The first 7 were already computed in~\cite[Theorem 1]{MovArticleMQ} and the other two are given below:
\begin{equation}\label{t_11}
s_{7} = 5^7(z-1)\frac{x_{01}x_{12}x_{23} - x_{01}x_{13}x_{22} - x_{02}x_{11}x_{23} + x_{02}x_{13}x_{21} + x_{03}x_{11}x_{22} - x_{03}x_{12}x_{21}}{x_{21}},
\end{equation}
\begin{multline}	
s_{8} = 5^7(z-1)\left(x_{01}x_{24} - x_{02}x_{23} + x_{03}x_{22} - x_{04}x_{21}\right)+ \\
+5^6z\left(x_{01}x_{22} + \frac52x_{01}x_{23} - x_{02}x_{21} - \frac52x_{03}x_{21}\right).
\end{multline}

\begin{proposition}
The Gauss-Manin connection restricted to the locus $L$ can be computed in terms of the coordinate $\tau_0$. It is given by:
\begin{equation}
\left.{\sf A}\right|_{L}=
\left(\begin{array}{ccccc}
0&0&0&0&0\\
0&0 & 1 &0 &0 \\
\frac{d\tau_5}{d\tau_0}&0 & 0 & \frac{d \tau_{3}}{d\tau_0} & 0 \\
0&0 & 0 & 0 & -1 \\
0&0 & 0 & 0 & 0
\end{array}\right)d\tau_0,
\end{equation}
where $\tau$ is given by~\eqref{matrixtau}.
\end{proposition}

\begin{proof}
To prove this, use the fact that the Gauss-Manin connection commutes with integrals, in the following sense:
$$
d\left(\int_{\delta}\omega\right) = \int_{\delta}\nabla\omega,
$$
where the integration on the right-hand side takes place only on $\HdR^3(X,Y)$ (recall that $\nabla\omega$ can be written as sum of elements of the form $\omega'\otimes s$, where $s$ a form in $\Omega^1_T$). Now, using this, we get:
$$
d\tau = \left[\int_{\delta_i}\nabla\alpha_j\right]_{i,j} = \left[\sum_k\int_{\delta_i}a_{jk}\alpha_k\right]_{i,j} =  \left[\sum_ka_{jk}\tau_{ik}\right]_{i,j} = \left[\sum_k\tau_{ik}a^{T}_{kj}\right]_{i,j} = \tau\cdot {\sf A}^T.
$$

This implies that the Gauss-Manin connection has to be given by
$$
\left.{\sf A}\right|_{L}=d \tau^{\mathrm{T}} \cdot \tau^{-\operatorname{T}}=\left(\begin{array}{ccccc}
0&0&0&0&0\\

-\tau_5d\tau_0 + d\tau_4&0 & d \tau_{0} & -\tau_{3} d \tau_{0}+d \tau_{1}&-\tau_{1} d \tau_{0}+\tau_{0} d \tau_{1}+d \tau_{2} \\
d\tau_5 & 0 & 0 & d \tau_{3} & -\tau_{3} d \tau_{0}+d \tau_{1} \\
0&0 & 0 & 0 & -d \tau_{0} \\
0&0 & 0 & 0 & 0
\end{array}\right)
$$

By using Griffths transversality and the fact that $\alpha_0$ is on $F^1$, we conclude that positions (2,1), (2,4) and (3,5) have to be zero, since our basis respect the Hodge filtration. After taking $d\tau_0$ out, we have the result. This yields the relations 
\begin{align}
\tau_1 &= -\frac{d\tau_2}{d\tau_0} -\tau_0\frac{d\tau_1}{d\tau_0},\\
\tau_3 &= \frac{d\tau_1}{d\tau_0},\\
\tau_5 &=\frac{d\tau_4}{d\tau_0}.
\end{align}
Notice the first two had already been computed in~\cite{MovArticleMQ}.
\end{proof}

By the uniqueness statement from Theorem~\ref{Ramanujan}, we conclude that $\frac{\partial}{\partial \tau_0}$ is the vector field $\sf R$ restricted to the locus $\sf L_{op}$. If we consider the coordinate $q = e^{2\pi i\tau_0}$, then $\sf R$ becomes $2\pi i q\frac{\partial}{\partial q}$. Writing the functions $\sf F$ and $\sf Y$ (or $\frac{d\tau_5}{d\tau_0}$ and $\frac{d\tau_3}{d\tau_0}$, respectively) in the coordinate $q$ gives us virtual counts of disks with boundary on the real quintic and virtual counts of rational curves on a quintic threefold. To see this, we just need to look at the expressions of the $\tau_i$ in terms of $x_{ij}$ we get, for example, that $\tau_4$ is given by $\frac{x_{01}}{x_{21}}$, i.e., the quotient of a solution for the non homogenous Picard-Fuchs equation by the holomorphic solution to the homogenous equation. This shows that the expressions for $F$ and $Y$ are the expressions in periods that we have for the disk potential and the Yukawa coupling.

\section{Further Extensions}

\subsection{Higher Genus Real GW invariants}
One application of our results is in the so-called BCOV theory, which is responsible for computing, at least conjecturally, higher genus Gromov-Witten invariants. In the closed case, the higher genus invariants were first computed in~\cite{huang_topological_2009} using techniques from~\cite{yamaguchi_topological_2004}. In \cite{movbook}, Movasati has shown that the generating function for higher genus are polynomials in the generators $t_i$, which generalize the classical fact that quasi-modular forms are polynomials in the Eisenstein series. 

In this section, we explain how to adapt this result for the real case, using the ideas from~\cite{alim-lange},~\cite{evidence-walcher} and~\cite{extended-walcher}. We state some results and conjectures, that will be addressed in future work. A more detailed version of this section is being prepared as a separated paper with P. Georgieva.

\subsubsection{Recollection of the closed case}

On the B-side, we can define the $\mathcal F^{(g)}$ as solutions to the so-called \textit{Holomorphic Anomaly Equation}. Even though they are not holomorphic, their expansion in the holomorphic limit (i.e., around the point $z=0$ on the space of parameters of the mirror quintic family -- see~\ref{mirror-quintic}) is conjecturally the generating functions for the genus $g$ Gromov-Witten invariants.For $g=0$ and $g=1$, the functions are computed directly, and for $g>1$, there is a recursive equation that determines the $\mathcal F^{(g)}$ up to a holomorphic term:
\begin{equation}\label{eq: hol-anomaly}
    \bar{\partial}_z \mathcal F^{(g)} = \frac12 \sum_{g_1+g_2=g}C_{\bar{z}}^{zz} \mathcal F_1^{(g_1)}\mathcal F_1^{(g_2)}+\frac12 C^{zz}_{\bar{z}}\mathcal F_2^{(g-1)},
\end{equation}

Here, $\mathcal F^{(g)}_i$ represents $D_i\mathcal F^{(g)}$, i.e., the covariant derivative applied to $\mathcal F$ and $C^{zz}_{\bar z} = \overline{C_{zzz}} G^{z\bar{z}} G^{z\bar{z}} e^{2K}$, where $K$ is the Kähler potential and $G$ is the metric in the space of parameters of the mirror quintic family (coordinate $(5\psi)^{-5} = z$, as used above).
$\mathcal F_i^{(g)}$ can be seen as a section of the bundle $Sym^n(TM)\otimes\mathbf L^{2g-2}$, where $\mathbf L$ is the vacuum bundle of holomorphic $(3,0)$-forms. The bundle is over the space of parameters of the mirror quintic family (i.e., the projective line minus three points). 

The solution of~\eqref{eq: hol-anomaly} was first given in~\cite{bershadsky_holomorphic_1993} and~\cite{bershadsky_kodaira-spencer_1994}. They used Feynman diagrams to solve the equations in terms of special functions called propagators, denoted by $S^{zz}$, $S^z$ and $S$. They satisfy
\[
\bar\partial_z S^{zz} = C_{\bar z}^{zz}\qquad\bar\partial_z S^{z} = G_{z\bar z}S^{zz}
\qquad\bar\partial_z S = G_{z\bar z}S^z 
\]

Yamaguchi and Yau~\cite{yamaguchi_topological_2004} have shown that, after multiplying the functions $\mathcal F^{(g)}$ by a holomorphic factor, they become polynomials in the propagators, $X = \frac 1{1-\psi^5}$ and the derivative of the Kähler potential, which corresponds, in the holomorphic limit, to the holomorphic solution of the Picard-Fuchs equation. In particular, the holomorphic ambiguity that cannot be computed from \eqref{eq: hol-anomaly} is a polynomial in the variable $X$. This last consideration allowed Huang, Klemm and Quackenbush~\cite{huang_topological_2009} to fix the coefficients of the ambiguities up to genus 51 by using properties of the expansions around $\psi = 0$ and $\psi = 1$ coming from physics.

Returning to our framework of the Gauss-Manin connection in disguise program, Movasati in \cite{movbook} considered non-holomorphic Calabi-Yau modular forms in Chapters 7, 8 and 9. The main result is Theorem 8 of \cite{movbook}, where it is showed that, in holomorphic limit, the functions $\mathcal F^{(g)}$ are polynomials in the generators $t_i$ defined from the moduli space $\mathsf T$ and that the anomaly equation~\eqref{eq: hol-anomaly} can be written in terms of seven unique vector fields $R_i$, which correspond to the Ramanujan vector field and to the generators of the Lie algebra of the closed version of the algebraic group $\mathsf{G}$ in~\eqref{coordinatesG}, i.e., without the extra coordinates $h_i$. See~\cite[3.10]{movbook}. 

\begin{proposition}[cf. {\cite[Thm. 8]{movbook}}] \label{prop:anomaly-vec}
    The anomaly equation~\eqref{eq: hol-anomaly} corresponds to the following equations in the moduli space $T$ of enhanced mirror quintics:
    \[
\begin{aligned}
 R_i\, F^g &= 0, \quad i=1,3,\\
 R_2\, F^g &= (2g-2)\, F^g,\\
 R_4\, F^g &= \tfrac{1}{2}\!\left( R_0^{2}\, F^{g-1}
 + \sum_{r=1}^{g-1} R_0 F^r\, R_0 F^{g-r}\right).
 \end{aligned}
 \]
 where $R_i$ are the unique vector fields for which the Gauss-Manin connection satisfies $\nabla_{R_i}(\omega) = A_{R_i}\omega = \mathfrak{g}_i$, where $\mathfrak{g}_i$ are generators of the Lie algebra of the group $\mathsf{G}$ which acts by changing the coordinates on the moduli space $T$, for $i=1,..,6$ and $R_0$ is the Ramanujan vector field associated to the Yukawa coupling (the closed version of Theorem~\ref{Ramanujan}).

\end{proposition}

In the propostion above, $F^g$ denotes the functions on the moduli space $\mathsf T$ that correspond to the $\mathcal F^{(g)}$. The proof of this proposition is purely computational, coming from the expressions of the propagators in terms of the $t_i$. However, it gives a description of the anomaly equation in the holomorphic limit in terms of the Gauss-Manin connection. 

\subsubsection{Real/Open BCOV Theory}

In this section we explain how to extend the results explained in the last section to the real Gromov-Witten invariants of higher genus. 

We start by considering the extended holomorphic anomaly equation as presented in~\cite{evidence-walcher}. We consider the functions $\mathcal G^{(\chi)}$, which, in the holomorphic limit, correspond conjecturally to the generating function of real Gromov-Witten invariants of Euler characteristic $\chi$, i.e., to virtual counts of J-holomorphic maps with boundary on the real Lagrangian of the quintic with source of Euler characteristic $\chi$. The equation reads
\begin{equation}\label{eq:ext-hol-anomaly}
\bar\partial_z\mathcal G^{(\chi)} = \frac12\sum_{\chi_1+\chi_2 =\chi-2}C^{zz}_{\bar z}\mathcal G^{(\chi_1)}_1\mathcal G^{(\chi_2)}_1+\frac12 C^{zz}_{\bar z}\mathcal G^{(\chi-2)}_2 - \Delta^{z}_{\bar z}\mathcal G_1^{(\chi-1)},
\end{equation}
where the subindices are applications of the covariant derivative as in~\ref{eq: hol-anomaly}, $\Delta^{z}_{\bar z}=\overline{\Delta_{zz}}G^{z\bar z}e^{K}$ and $\Delta_{zz}$ is the disk potential, that in the holomorphic limit correspond to the expansion $\mathsf F$ in Theorem \ref{Ramanujan}.

Notice that it is not enough to consider orientable Riemann surfaces as was done in the paper~\cite{extended-walcher}. Although this lead to interesting functions $\mathcal F^{(g,h)}$ that satisfy the same equation~\ref{eq:ext-hol-anomaly} they do not correspond to counts of real curves. The rigorous construction of real Gromov-Witten invariants in~\cite{georgieva_real_2018} shows the need for considering nonorientable curves. This also implies that a integral BPS expansion is only expected for $\mathcal G$. 

A Feynmann diagram solution for equation~\ref{eq:ext-hol-anomaly} is also possible in terms of the same propagators as in the closed case and two extra functions called terminators. This was only written down explicitly for the oriented case (see~\cite{extended-walcher} and~\cite{konishi_solutions_2007}), but the computation carries over in the real case by changing the initial conditions. 

In the context of the Gauss-Manin connection in disguise proposed in this paper, we prove the following:

\begin{proposition}
    In the coordinates and notation of~\cite{alim-lange}, we have the following polynomials
    \begin{align}
        z^3P^{-1}\mathcal G^{(-2)}&=5,\\
        z^2P^{-\frac{1}2}\mathcal G^{(-1)}&=-5\mathcal E^z,\\
        z\mathcal G^{(0)}&=-\frac{55}{48} + \frac54(\mathcal E^z)^2 + \frac{5 P}{48} + \frac{25}6{\theta K} + \frac{T^{zz}}2, 
    \end{align}
    each expressed in the propagators $T^{zz}$ (degree $1$),$T^z$ (degree $2$) and $T$ (degree $3$), the terminators $\mathcal E^z$ (degree $\frac 12$) and $\mathcal E$ (degree $\frac 32$), the derivative of the Kähler parameter $\theta K$ (degree 1) and $p$ (degree $\frac 12$), where $p^2 = P-1$ and $P = \frac{1}{1-5^5z}$.

    In general, $z^n P^{\frac{\chi}2}\mathcal G^{(\chi)}_n$ is a polynomial of degree $\frac{3\chi}2+n$ in the variables above.
    \end{proposition}
\begin{proof}
    The first three equations are just a direct computation from the equations in~\cite{evidence-walcher} and from the closed potentials in~\cite{huang_topological_2009}. The second part is a simple application of the anomaly equation. Using that $\bar\partial T^{zz}=\frac{5P}{z^2}C^{zz}_{\bar z}$, $\bar\partial T^{z}=T^{zz}\cdot zG^{z\bar z}$, $\bar\partial T=T^{z}\cdot zG^{z\bar z}$, $\bar\partial \mathcal E^{z}= \frac{P^{\frac12}}z\Delta_{\bar z}^z$, $\bar\partial \mathcal E= \mathcal E^{z} zG^{z\bar z}$ and $\bar\partial \theta K = zG^{z\bar z}$ and collecting the terms $C^{zz}_{\bar z}$, $\Delta^{z}_{\bar z}$ and $G^{z\bar z}$, we get:
    \begin{align}
        \frac{5P}{z^2}\frac{\partial\mathcal G^{(\chi)}}{\partial T^{zz}}=\frac12\left(\sum_{\chi_1+\chi_2 =\chi-2}\mathcal G^{(\chi_1)}_1 \mathcal G^{(\chi_2)}_1+\mathcal G^{(\chi-2)}_2\right)\\
        \frac{P^{\frac12}}z\frac{\partial\mathcal G^{(\chi)}}{\partial \mathcal E^z} =-\mathcal G_1^{(\chi-1)}\\
       \frac{\partial\mathcal G^{(\chi)}}{\partial \theta K}+T^{zz}\frac{\partial\mathcal G^{(\chi)}}{\partial T^z}+ T^{z}\frac{\partial\mathcal G^{(\chi)}}{\partial T}+\mathcal E^z\frac{\partial\mathcal G^{(\chi)}}{\partial \mathcal E}=0
    \end{align}

    After multiplying the first equation by $z^2P^{\frac{\chi-2} 2}$, the second, by $zP^{\frac{\chi-1}2}$, and the third, by $P^{\frac{\chi}2}$, we get:
    \begin{align}
        5\frac{\partial P^{\frac{\chi} 2}\mathcal G^{(\chi)}}{\partial T^{zz}}=\frac12\left(\sum_{\chi_1+\chi_2 =\chi-2}zP^{\frac{\chi_1}2}\mathcal G^{(\chi_1)}_1 z P^{\frac{\chi_2}2}\mathcal G^{(\chi_2)}_1+z^2P^{\frac{\chi-2}2}\mathcal G^{(\chi-2)}_2\right)\\
        \frac{\partial P^{\frac{\chi}2}\mathcal G^{(\chi)}}{\partial \mathcal E^z} =-zP^{\frac{\chi-1}2}\mathcal G_1^{(\chi-1)}\\
       \frac{\partial P^{\frac{\chi}2}\mathcal G^{(\chi)}}{\partial \theta K}+ T^{zz}\frac{\partial P^{\frac{\chi}2}\mathcal G^{(\chi)}}{\partial T^z}+ T^z\frac{\partial P^{\frac{\chi}2}\mathcal G^{(\chi)}}{\partial T}+\mathcal E^z \frac{\partial P^{\frac{\chi}2}\mathcal G^{(\chi)}}{\partial \mathcal E}=0
    \end{align}
    
    Using induction, we conclude that the right-hand sides of the first two equations are polynomials in the variables above. This implies that, after integrating, $P^{\frac\chi2}\mathcal G^{(\chi)}$ is a polynomial up to holomorphic part.
    
    Now, from the asymptotic behavior, the holomorphic ambiguity has the same form as in~\cite[Sec. 2.4]{konishi_solutions_2007}, 
    \begin{align}
     f^{(\chi)} = \frac{\sum_{i=0,\, i\, {\rm odd}}^{3\chi} a_i z^{\frac{i}2}}{(1-5^5z)^{\chi}},\, {\rm for}\, \chi\, {\rm odd}\\ 
     f^{(\chi)} = \frac{\sum_{i=0,\, i\, {\rm even}}^{3\chi} a_i z^{\frac{i}2}}{(1-5^5z)^{\chi}},\, {\rm for}\, \chi\, {\rm even}\\
    \end{align}
    which, after multiplication by $P^{\frac{\chi}2}$, can be expressed in terms of $P=p^2+1$ and $\sqrt{P}\sqrt{5^5z}=\sqrt{P-1} = p$.
\end{proof}
One can notice that the variables $s_7$ and $s_8$ are directly related to the terminators $\mathcal E^z$ and $\mathcal E$ (see~\ref{Ramanujan} and the definition of the variables in \cite{alim-lange}). We therefore expect to be able to write the functions $\mathcal G$ as functions on the moduli space $S$ and prove a result analogous to~\ref{prop:anomaly-vec}. We, however, leave this as topic for future work.

\begin{conjecture}
The extended anomaly equation~\ref{eq:ext-hol-anomaly} corresponds to vector fields on the moduli space $S$ that form the Lie algebra of the group $\mathsf G$ of base change defined in \ref{coordinatesG}.
\end{conjecture}

\subsection{Moving families}

In conclusion, we discuss another possible extension of our work, namely, considering moving families instead of a fixed pair of conics. It is natural to consider a family of divisors in the mirror quintic and see if we get a generalization of Jacobi forms as in the elliptic curve case~\cite{MVC}. A nice starting point would be the paper~\cite{jockers-soroush}, in which the authors consider families of divisors on the mirror quintic from a Physical point of view. 

There, they also consider mixed Hodge structures, but on the relative cohomology $H^3(X_{\psi},V_{\phi})$, where $X$ is a mirror quintic~\ref{mirror-quintic} and $V$ is the family of divisors given by the equation $x_4^4-\phi x_0x_1x_2x_3=0$. The cohomology group in this case is 7-dimensional: the holomorphic 3-form and its derivatives that form a basis for $H^3(X)$ (see~\ref{proof1}) and the derivative with respect to $\phi$, and the mixed derivatives.
\[
(\omega_1, \partial_\psi \omega_1, \partial_\psi^2 \omega_1, \partial_\psi^3 \omega_1, \partial_\phi \omega_1, \partial_\psi \partial_\phi \omega_1, \partial_\psi^2 \partial_\phi \omega_1)
\]

If one considers the corresponding moduli space of enhanced triples $(X_{\psi},V_{\phi},\alpha)$, where $\alpha$ is a basis of the cohomology respecting the MHS with constant intersection product, we expect to have a space of higher dimension. Considering that there is one dimension for the choice of $X$, one for $V$, and one for $\omega_1$. For the Hodge structure, we would have 2 elements of the basis in $F^2\setminus F^3$, 2 elements in $F^1\setminus F^2$, and 2 elements in $F^0\setminus F^1$. The weight filtration will come from the Hodge structure in $H^2(V)$, similar to what we saw in section \ref{MHS and GM}. The main difficulty is the computation of the intersection product that will give further relations between the variables. This involves understanding the coupling
$$
\int \omega_1\wedge\partial_{\phi}\partial^2_{\psi}\omega_1
$$
and its derivatives. 

This should be possible after a careful analysis from the operators $\mathcal L_i$, $i=1,2,3$ in section 5.2 of \cite{jockers-soroush}. We conjecture that such coupling should be related to the inverse of the discriminant $D_2=\phi(\phi-5\psi)^4-256$ (see equation (5.13) in~\cite{jockers-soroush}), in a way similar to how the inverse of $D_1 = 1-\psi^5$ is related to the Yukawa coupling for the closed situation.

We leave the details of this extension for a future paper.

\bigskip

{\bf Acknowledgments.} I would like to thank my advisor Hossein Movasati for helping me through the development of this work and for presenting this problem to me. I also would like to thank Johannes Walcher, Martin Vogrin, Hans Jockers, Lukas Hahn and Sebastian Nils for valuable discussions and comments on early versions of this work. I also would like to thank IMPA for the infrastructure and research ambient and CNPq for financial support via grant 141069/2020-1. Finally, I would like to thank the Institute of Mathematics of the Heidelberg University and the Mainz Institute for Theoretical Physics for the great reception they provided during my visit. 

Section 5 was added almost three years after the paper first appeared on the ArXiv. I would like to thank the anonymous referees for helpful suggestions that improved the paper. I would also like to thank P. Georgieva and J. Walcher for helpful discussions regarding the last section. Finally, I acknowledge the support of the European Research Council through the grant ROGW-864919 for this last part of the work.

{\bf Data Availability}. The datasets (algorithms for computations and its results) generated during and/or analysed during the current study are available from the corresponding author on reasonable request. All the computations were done using the software MAGMA.

{\bf Conflict of Interest.} The author have no competing interests to declare that are relevant to the content of this article.

\bibliography{ArtigoGMCDOpen.bib}

\begin{thebibliography}{CDLOGP91}

\bibitem[AKV22]{alimkurylenkovogrin}
Murad Alim, Vadym Kurylenko, and Martin Vogrin.
\newblock The algebra of derivations of quasi-modular forms from mirror symmetry.
\newblock {\em Pure and Applied Mathematics Quarterly}, 18(3):1037--1073, July 2022.
\newblock Publisher: International Press of Boston.

\bibitem[AL07]{alim-lange}
Murad Alim and Jean~Dominique Länge.
\newblock Polynomial structure of the (open) topological string partition function.
\newblock {\em Journal of High Energy Physics}, 2007(10):045, October 2007.

\bibitem[AMSY16]{AMSY}
Murad Alim, Hossein Movasati, Emanuel Scheidegger, and Shing-Tung Yau.
\newblock Gauss–{Manin} {Connection} in {Disguise}: {Calabi}–{Yau} {Threefolds}.
\newblock {\em Commun. Math. Phys.}, 344(3):889--914, June 2016.

\bibitem[BCOV93]{bershadsky_holomorphic_1993}
M.~Bershadsky, S.~Cecotti, H.~Ooguri, and C.~Vafa.
\newblock Holomorphic anomalies in topological field theories.
\newblock {\em Nuclear Physics B}, 405(2):279--304, September 1993.

\bibitem[BCOV94]{bershadsky_kodaira-spencer_1994}
M.~Bershadsky, S.~Cecotti, H.~Ooguri, and C.~Vafa.
\newblock Kodaira-{Spencer} theory of gravity and exact results for quantum string amplitudes.
\newblock {\em Communications in Mathematical Physics}, 165(2):311--427, October 1994.

\bibitem[CDLOGP91]{CdlOGP}
Philip Candelas, Xenia~C. De~La~Ossa, Paul~S. Green, and Linda Parkes.
\newblock A pair of {Calabi}-{Yau} manifolds as an exactly soluble superconformal theory.
\newblock {\em Nuclear Physics B}, 359(1):21--74, July 1991.

\bibitem[CK99]{katz-cox}
David~A. Cox and Sheldon Katz.
\newblock {\em Mirror symmetry and algebraic geometry}.
\newblock Mathematical surveys and monographs 68. American Mathematical Society, 1999.

\bibitem[CMVL24]{MVC}
Jin Cao, Hossein Movasati, and Roberto Villaflor~Loyola.
\newblock Gauss–{Manin} {Connection} in {Disguise}: {Quasi} {Jacobi} {Forms} of {Index} {Zero}.
\newblock {\em International Mathematics Research Notices}, 2024(8):6680--6709, April 2024.

\bibitem[GP90]{greene-plesser}
B.R. Greene and M.R. Plesser.
\newblock Duality in {Calabi}-{Yau} moduli space.
\newblock {\em Nuclear Physics B}, 338(1):15--37, July 1990.

\bibitem[Gro66]{Grothendieck}
Alexander Grothendieck.
\newblock On the de {Rham} cohomology of algebraic varieties.
\newblock {\em Publications Mathématiques de l'IHÉS}, 29:95--103, 1966.

\bibitem[GZ18]{georgieva_real_2018}
Penka Georgieva and Aleksey Zinger.
\newblock Real {Gromov}-{Witten} theory in all genera and real enumerative geometry: {Construction}.
\newblock {\em Annals of Mathematics}, 188(3), November 2018.

\bibitem[HKQ09]{huang_topological_2009}
M.-x. Huang, A.~Klemm, and S.~Quackenbush.
\newblock Topological {String} {Theory} on {Compact} {Calabi}–{Yau}: {Modularity} and {Boundary} {Conditions}.
\newblock In {\em Homological {Mirror} {Symmetry}: {New} {Developments} and {Perspectives}}, pages 1--58. Springer, Berlin, Heidelberg, 2009.

\bibitem[JS09]{jockers-soroush}
Hans Jockers and Masoud Soroush.
\newblock Effective {Superpotentials} for {Compact} {D5}-{Brane} {Calabi}-{Yau} {Geometries}.
\newblock {\em Commun. Math. Phys.}, 290(1):249--290, August 2009.

\bibitem[KM07]{konishi_solutions_2007}
Yukiko Konishi and Satoshi Minabe.
\newblock On solutions to {Walcher}’s extended holomorphic anomaly equation.
\newblock {\em Communications in Number Theory and Physics}, 1(3):579--603, 2007.

\bibitem[KO68]{katz-oda}
Nicholas~M. Katz and Tadao Oda.
\newblock On the differentiation of {De} {Rham} cohomology classes with respect to parameters.
\newblock {\em Journal of Mathematics of Kyoto University}, 8(2):199--213, January 1968.
\newblock Publisher: Duke University Press.

\bibitem[Mov12]{mov-elliptic}
Hossein Movasati.
\newblock Quasi-modular forms attached to elliptic curves, {I}.
\newblock {\em Annales Mathématiques Blaise Pascal}, 19(2):307--377, 2012.

\bibitem[Mov15]{MovArticleMQ}
Hossein Movasati.
\newblock Modular-type functions attached to mirror quintic {Calabi}–{Yau} varieties.
\newblock {\em Math. Z.}, 281(3-4):907--929, December 2015.

\bibitem[Mov17]{movbook}
Hossein Movasati.
\newblock {\em Gauss–{Manin} {Connection} in {Disguise}: {Calabi}–{Yau} {Modular} {Forms}}.
\newblock International Press of Boston, Incorporated, Somerville, April 2017.

\bibitem[MW09]{mor-wal}
David~R. Morrison and Johannes Walcher.
\newblock D-branes and normal functions.
\newblock {\em Advances in Theoretical and Mathematical Physics}, 13(2):553--598, April 2009.
\newblock Publisher: International Press of Boston.

\bibitem[PS08]{peters-steenbrink}
Christiaan Peters and Joseph H.~M. Steenbrink.
\newblock {\em Mixed {Hodge} {Structures}}.
\newblock Ergebnisse der {Mathematik} und ihrer {Grenzgebiete}. 3. {Folge} / {A} {Series} of {Modern} {Surveys} in {Mathematics}. Springer-Verlag, Berlin Heidelberg, 2008.

\bibitem[PSW08]{pand-wal-sol}
R.~Pandharipande, J.~Solomon, and J.~Walcher.
\newblock Disk {Enumeration} on the {Quintic} 3-{Fold}.
\newblock {\em Journal of the American Mathematical Society}, 21(4):1169--1209, 2008.

\bibitem[Wal07]{walcher-opening}
Johannes Walcher.
\newblock Opening {Mirror} {Symmetry} on the {Quintic}.
\newblock {\em Commun. Math. Phys.}, 276(3):671--689, December 2007.

\bibitem[Wal09a]{evidence-walcher}
Johannes Walcher.
\newblock Evidence for tadpole cancelation in the topological string.
\newblock {\em Communications in Number Theory and Physics}, 3(1):111--172, 2009.

\bibitem[Wal09b]{extended-walcher}
Johannes Walcher.
\newblock Extended holomorphic anomaly and loop amplitudes in open topological string.
\newblock {\em Nuclear Physics B}, 817(3):167--207, August 2009.

\bibitem[YY04]{yamaguchi_topological_2004}
Satoshi Yamaguchi and Shing-Tung Yau.
\newblock Topological string partition functions as polynomials.
\newblock {\em Journal of High Energy Physics}, 2004(07):047, August 2004.

\end{thebibliography}

\end{document}